\documentclass[12pt,a4paper]{amsart}
 \pdfoutput=1
 
 \usepackage{amsmath,amssymb,amsthm}
 \usepackage{bm}  
 \usepackage{mathtools}
 \usepackage{tikz}
 \usepackage{thmtools}
 \usepackage{float} 
 
 \usepackage{xcolor} 	
 \usepackage{hyperref}
 \hypersetup{
 	colorlinks,
     linkcolor={red!60!black},
     citecolor={green!60!black},
     urlcolor={blue!60!black},
 }
 \usepackage[maxbibnames=99]{biblatex}
 \usepackage[utf8]{inputenc}
 \usepackage[T1]{fontenc}
 \usepackage{lmodern}
 \usepackage[babel]{microtype}
 \usepackage[english]{babel}
 
 \usepackage{geometry}
 \usepackage{enumitem}

\theoremstyle{plain}
\newtheorem{thm}{Theorem}[section]

\newtheorem{prop}[thm]{Proposition}

\newtheorem{claim}[thm]{Claim}
\newtheorem{cor}[thm]{Corollary}
\newtheorem{lem}[thm]{Lemma}

\newtheorem{quest}[thm]{Question}

\newtheorem*{thm*}{Theorem}
\newtheorem*{cor*}{Corollary}

\theoremstyle{definition}
\newtheorem{rem}[thm]{Remark}
\newtheorem{dfn}[thm]{Definition}

\title{On  generalisations of the Aharoni-Pouzet base exchange theorem}

\author{Zsuzsanna  Jankó}
\thanks{Jankó is grateful for the support of  NKFIH 
OTKA-K128611}
\address{Zsuzsanna  Jankó,  Department of Operations Research and Actuarial Sciences,  Corvinus University of Budapest,  
Budapest, Hungary;
Institute of Economics, Centre for Economic and Regional Studies, 
Budapest, Hungary}
\email{zsuzsanna.janko@uni-corvinus.hu, janko.zsuzsanna@krtk.hu}
\author{Attila Jo\'{o}}
\thanks{Jo\'{o} would like to thank the generous support of the Alexander 
von Humboldt Foundation and NKFIH 
OTKA-129211}
\address{Attila Jo\'{o},
Department of Mathematics, University of Hamburg,  Hamburg, Germany; Set theory, logic and topology research division, 
Alfr\'{e}d R\'{e}nyi Institute of Mathematics,  Budapest, Hungary}
\email{attila.joo@uni-hamburg.de, jooattila@renyi.hu}

\keywords{base exchange, infinite matroid}
\subjclass[2020]{Primary: 05B35, 15A03, 03E05,    Secondary: 05A18, 05C05, 05B40} 
\begin{document}

\begin{abstract}
The Greene-Magnanti theorem states that if $ M $ is a finite matroid, $ B_0 $ and $ B_1 $ are bases and  $ 
B_0=\bigcup_{i=1}^{n} X_i $ is a partition, then there is a partition $ B_1=\bigcup_{i=1}^{n}Y_i $  such that $ (B_0 
\setminus X_i) 
\cup Y_i $ is a base 
for every $ i $. The 
special case where each $ X_i $ 
is a singleton can be rephrased as the existence of a perfect matching in the base transition graph. Pouzet conjectured that this remains true in 
infinite dimensional vector spaces. Later he and Aharoni answered this conjecture affirmatively not just for vector spaces but for infinite matroids. 

We prove two generalisations of their result.  On the one hand, we show that `being a singleton' can be relaxed to `being 
finite'  and this is sharp in the sense the exclusion of infinite sets is really necessary.  
On the other hand,  we prove that if $ B_0$ and $ B_1 $ are bases, then  
there is a bijection $ F $ between their finite subsets such that $ (B_0\setminus I) \cup F(I) $ is a 
base for every  $ I$. In contrast to the approach of Aharoni and Pouzet, our proofs are completely elementary, they do not rely on infinite matching 
theory.
\end{abstract}
\maketitle

\section{Introduction}
In the usual axiomatization of finite matroids in the terms of bases, one of the axioms demands that if $ B_0 $ and $ B_1 $ 
are bases, then for every $ x\in B_0 $, there is a  $ y\in B_1 $ such that $ B_0-x+y $ is a base. An important research direction in matroid theory 
is looking for stronger base exchange properties.  Let us mention a few fundamental results in this subfield. First of all, the $ y $ above can be 
chosen in such a way that the exchange is ``symmetric'' in the sense that  $ B_1-y+x $ is also a base. Greene's theorem \cite{greene1973multiple} 
is a strengthening of this symmetric base exchange property stating that symmetric exchange of subsets of bases is also possible. 
Namely, for every $ X\subseteq B_0 $ there is a $ Y\subseteq B_1 $ such that $ (B_0\setminus X) \cup Y$ and $ (B_1\setminus Y)\cup X $ are 
both bases. The Greene-Magnanti theorem \cite{greene1975some} states that the following partition base exchange property also holds: If 
$B_0=\bigcup_{i=1}^{n} X_i 
$ is a partition, then  there is a partition $B_1=\bigcup_{i=1}^{n} Y_i $ such that $ (B_0\setminus X_i)\cup Y_i $ is a base of every $ i $. Note 
that Greene's theorem is equivalent with the   special case $ n=2 $ of this. A more recent result by Koltar, Roda 
and R. Ziv \cite[Theorem 1.1.]{koltar2021seq} ensures that in the setting of the Greene-Magnanti theorem one can choose the sets $ Y_i $ is such 
a way that not 
just $ (B_0\setminus 
X_i)\cup Y_i $ but also $ (B_0 \setminus \bigcup_{j=1}^{i}X_j)\cup \bigcup_{j=1}^{i}Y_i $ is a base for every $ i $.

Pouzet initiated the investigation of base exchange properties of infinite-dimensional 
vector spaces. In particular, he was interested in if the special case of the Greene-Magnanti theorem where all the sets $ X_i $ are singletons 
remains 
true in this more general setting.  This had been settled affirmatively by Aharoni and Pouzet \cite[Theorem 2.1]{aharoni1991bases} in a more 
general 
context provided by Definition 
\ref{def: finitary matroid}. 

Vector spaces are the 
motivating examples of matroids but several important vector spaces are infinite or even has infinite dimension. This led to the following matroid 
concept where the ground set and the bases are allowed to be infinite:

\begin{dfn}\label{def: finitary matroid}
A finitary matroid is a pair  $M=(E,\mathcal{I})$ with ${\mathcal{I} \subseteq \mathcal{P}(E)}$ such that  
\begin{enumerate}
[label=(\Roman*)]
\item\label{item: matroid axiom1} $ \emptyset\in  \mathcal{I} $;
\item\label{item: matroid axiom2} $ \mathcal{I} $ is downward closed;
\item\label{item: matroid axiom3} If $ I,J\in \mathcal{I} $ with $ \left|I\right|<\left|J\right| $, then there exists an $  e\in J\setminus I $ such that
$ I+e\in \mathcal{I} $;
\item\label{item: matroid axiom4} If all the finite subsets of  an infinite set $ X $ are in $ \mathcal{I} $, then $ X\in \mathcal{I} $.
\end{enumerate} 
\end{dfn}
\begin{rem}\
\begin{itemize}
\item If~$E$ is finite, then  \ref{item: matroid axiom1}-\ref{item: matroid axiom3}  is  the usual axiomatization of finite 
matroids in terms 
of independent sets while \ref{item: matroid axiom4} is redundant.
\item It is enough to demand axiom \ref{item: matroid axiom3} for finite $ I $ and $ J $.
\item Some authors call the concept defined in Definition \ref{def: finitary matroid} simply ``matroid'' (see for example 
\cite{polat1987compactness, komjath1988compactness, aharoni1998intersection}), 
other authors refer to it  as ``independence structure'' \cite{mcdiarmid1975exchange} or ``independence space'' \cite{jerzy2005matr} but 
the term ``finitary matroid'' 
became dominant in the literature.
\item The word ``finitary'' reflects to the fact that by axiom \ref{item: matroid axiom4} every circuit (i.e. minimal dependent set) is finite. This is 
in contrast to the more general concept of matroids discovered independently by Higgs \cite{higgs1969matroids} and Bruhn et al. 
\cite{bruhn2013axioms} where 
``infinitary matroids'' (i.e. 
matroids with infinite circuits) also exist.
\item Duals of finitary matroids are usually infinitary which was originally the main motivation of Rado to express the need  
\cite{rado1966abstract} for a more general infinite matroid concept than finitary matroids. 
\end{itemize}
\end{rem}
For more information about infinite matroids we recommend the chapter ``Infinite Matroids'' by Oxley in \cite{white1992matroid} and the 
habilitation thesis   \cite{nathanhabil} of Bowler with the same title.

On the one hand, we give an example that Greene's theorem may fail even for finite-cycle 
matroids of infinite graphs:
\begin{restatable}{thm}{GreeneFail}\label{thm: Greene fails}
There is a countably infinite graph $ G=(V, E) $ with edge-disjoint spanning trees $ T_0$ and $T_1 $ such that there is a bipartition $ 
E(T_0)=X_0\cup X_1 $ for which there are no edge-disjoint spanning trees $ S_0 $ and $ S_1 $ with $ E(T_0)\cap E(S_i)=X_i $ for $ i\in \{ 0,1 
\} $.
\end{restatable}
On the 
other hand, we prove two  generalisations of the Aharoni-Pouzet base exchange theorem:
\begin{thm}\label{thm: GreeneMagnantiIntro}
Suppose that $ M=(E, \mathcal{I}) $ is a finitary matroid, $ B_0 $ and $ B_1 $ are 
bases of $ M $  and $B_0=\bigcup_{i<\kappa}X_i  $ is a 
partition where each $ X_i $ is finite. Then there is a 
partition $B_1=\bigcup_{i<\kappa}Y_i  $ such that $ (B_0\setminus X_i)\cup Y_i $ is a base for 
each $ i<\kappa $.
\end{thm}

\begin{restatable}{thm}{ExchangeAllFin}\label{thm: ExchangeAllFin}
Suppose that $ M=(E, \mathcal{I}) $ is a finitary matroid and $ B_0 $ and $ B_1 $ are 
bases of 
$ M $. Then 
there is a bijection $ F: 
[B_0]^{<\aleph_0}\rightarrow  [B_1]^{<\aleph_0}$ such that  $ (B_0\setminus I)\cup F(I) $ is a 
base for every $I\in [B_0]^{<\aleph_0} $.
\end{restatable}

In the following section we introduce a few notation, then in Section \ref{sec: counterex} we present our counterexample in Theorem \ref{thm: 
Greene fails}. Finally,  the last section (Section \ref{sec: positiveBaseExch}) is devoted to the proofs of the positive results Theorems \ref{thm: 
GreeneMagnantiIntro} and 
\ref{thm: ExchangeAllFin}.

\section{Basic definitions and notation}
We use some standard set-theoretic notation, in particular: the variable $ \kappa $ stands for cardinal numbers, $ \alpha $ and $ \beta $ are ordinals, 
the set of natural numbers is denoted by $ \omega $, we write $ [X]^{<\kappa} $ for the set of subsets of $ X $ of size less than $ \kappa $, 
functions are represented as sets of ordered pairs. 

A \emph{matroid} is an ordered pair  $M=(E,\mathcal{I})$ with ${\mathcal{I} \subseteq 
\mathcal{P}(E)}$ such that  
\begin{enumerate}
[label=(\Roman*)]
\item $ \emptyset\in  \mathcal{I} $;
\item $ \mathcal{I} $ is downward closed;
\item[{(III')}]\label{item: matroid axiom3'} For every $ I,J\in \mathcal{I} $ where  $J $ is $ \subseteq $-maximal in $ \mathcal{I} $ and $ I 
$ is 
not, there exists an $  e\in J\setminus I $ such that
$ I+e\in \mathcal{I} $;
\item[{(IV')}]\label{item: matroid axiom4'} For every $ X\subseteq E $, any $ I\in \mathcal{I}\cap 
\mathcal{P}(X)  $ can be extended to a $ \subseteq $-maximal element of 
$ \mathcal{I}\cap \mathcal{P}(X) $.
\end{enumerate}

The sets in~$\mathcal{I}$ are called \emph{independent} while the sets in ${\mathcal{P}(E) \setminus \mathcal{I}}$ are 
\emph{dependent}. The maximal independent sets are called \emph{bases}. The \emph{rank} of a matroid is the size of its bases. The minimal 
dependent sets are the \emph{circuits}. Every dependent set contains a circuit (which is non-trivial for infinite matroids). The class of finitary 
matroids (see Definition \ref{def: finitary 
matroid}) consist exactly of those matroids in which every circuit is finite. For an  ${X \subseteq E}$, the pair ${\boldsymbol{M  
\upharpoonright X} :=(X,\mathcal{I} 
\cap \mathcal{P}(X))}$ is a matroid and it is called the \emph{restriction} of~$M$ to~$X$. We write ${\boldsymbol{M - X}}$ for 
$ M  \upharpoonright (E\setminus X) $  and call it the minor obtained by the 
\emph{deletion} of~$X$. 
The \emph{contraction} of $ X $ in $ M $ is a matroid on $ E\setminus X $ in which $ I\subseteq E\setminus X $ is independent iff $ J\cup I $ is 
independent in $ M $ for a (equivalently: for every) maximal independent subset $ J $ of $ X $.
Contraction and deletion commute, i.e. for 
disjoint 
$ X,Y\subseteq E $, we have $ (M/X)-Y=(M-Y)/X $.  Matroids of this form are the  \emph{minors} of~$M$. We say~${X 
\subseteq E}$ \emph{spans}~${e \in E}$ in matroid~$M$ if either~${e \in X}$ or $ \{ e \} $ is dependent in $ M/X $. 
If $ I $ is independent in $ M $  but $ I+e $ is dependent for some $ e\in E\setminus I $,  then there is a unique 
circuit   $ \boldsymbol{C_M(e,I)} $ of $ M $ through $ e $ contained in $ I+e $ which is called the \emph{fundamental circuit} of $ e $ on $ I $.

\section{The failure of Greene's base exchange theorem in infinite matroids}\label{sec: counterex}

Assume that $ M $ is a finitary matroid, $ B_0 $ and $ B_1 $ are disjoint bases and $ B_0=X_0\cup X_1 $ is a bipartition. It was shown by 
McDiarmid in \cite{mcdiarmid1975exchange} that there are disjoint independent sets $ I_0 $ and $ I_1 $ with $ I_0\cup I_1=B_0\cup B_1 $ and 
$ B_0 \cap I_i=X_i $ for $ i\in \{ 0,1 \} $. If $ M $ has a finite rank, then the $ I_i $ need to be bases because they have together $ 2\cdot r(M) $ 
elements. This argument fails if $ r(M) $ is infinite. Is it still true that they need to be bases? If not, is it always  possible to choose them to 
be bases? We demonstrate the negative answer for these questions:

\GreeneFail*
\begin{proof}
To simplify the notation let us identify the spanning trees with their edge sets.
Let $ G=(V,E) $ be the graph at Figure \ref{fig: GreeneFails}.

\begin{figure}[H]
\centering

\begin{tikzpicture}

\node[inner sep=0pt] (v1) at (-4,0) {$v_0$};
\node[inner sep=0pt] (v2) at (-2,0) {$v_1$};
\node[inner sep=0pt] (v3) at (-0,0) {$v_2$};
\node[inner sep=0pt] (v4) at (2,0) {$v_3$};
\node[inner sep=0pt] (v5) at (4,0) {$v_4$};
\node[inner sep=0pt] (v6) at (6,0) {$v_5$};
\node[inner sep=0pt] (v7) at (8,0) {$v_6$};
\node[inner sep=0pt] (v8) at (10,0) {$v_7$};
\node[inner sep=0pt] (v9) at (10.5,0) {$\dots$};

\draw  (v1) edge (v2);
\draw  (v2) edge[dashed] (v3);
\draw  (v3) edge (v4);
\draw  (v4) edge[dashed] (v5);
\draw  (v5) edge (v6);
\draw  (v6) edge[dashed] (v7);
\draw  (v7) edge (v8);

\draw  (v1) edge[dashed, out=-45,in=225,looseness=1]  (v4);
\draw  (v3) edge[dashed, out=-45,in=225,looseness=1]  (v6);
\draw  (v5) edge[dotted, out=-45,in=225,looseness=1]  (v8);
\draw  (v2) edge[out=45,in=135,looseness=1]  (v4);
\draw  (v4) edge[out=45,in=135,looseness=1]  (v6);
\draw  (v6) edge[dotted, out=45,in=135,looseness=1]  (v8);

\node at (-1.2,-1.8) {$h_0$};
\node at (3.2,-1.8) {$h_1$};
\node at (7.2,-1.8) {$h_2$};
\node at (-3,0.2) {$f_0$};
\node at (-1,0.2) {$f_1$};
\node at (1,0.2) {$f_2$};
\node at (3,0.2) {$f_3$};
\node at (5,0.2) {$f_4$};
\node at (7,0.2) {$f_5$};
\node at (9,0.2) {$f_6$};
\node at (0,1.2) {$e_0$};
\node at (4,1.2) {$e_1$};
\node at (8,1.2) {$e_2$};
\node at (5.2,-4.8) {};
\end{tikzpicture}
\vspace*{-2.5cm}
\caption{The graph $ G $. Right after the step $ n=1 $ of the induction we already know that  normal edges must belong to $ 
S_0 $ and dashed edges are in $ S_1 $.  The 
affiliation of the dotted edges are
unknown at this point.} \label{fig: GreeneFails}

\end{figure}
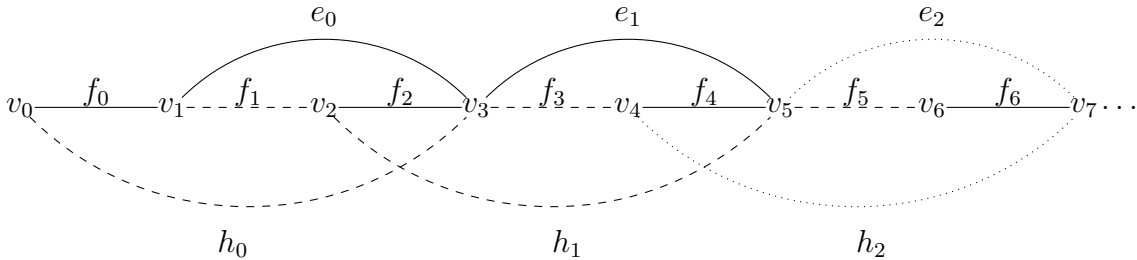

Let us denote the set of edges with exactly one end-vertex in $ U\subseteq V $ by $ \delta(U) $. Consider the spanning trees $ T_0:=\{ f_{n}:\ 
n<\omega \}$ and $ T_1:=E\setminus T_0 $. We  define 
$ X_0:=\{ f_{2n}:\ n<\omega \} $ 
and $ X_1:=\{ f_{2n+1}:\ n<\omega \} $. Suppose for a contradiction that $ S_0 $ and $ S_1 $ are edge-disjoint spanning trees with $ T_0\cap 
S_i=X_i $ for $ i\in \{ 0,1 \} $.  We must have $ h_0\in S_1 $  since 
otherwise $ v_0 $ were an 
isolated vertex in $ S_1 $.  But then all the edges in 
$ \delta(\{ v_0,\ v_1 \}) $ but $ e_0 $ are in $ S_1 $, thus we must have $ 
e_0\in S_0 $. Suppose that we already know for some $ n<\omega $ that $ \{ h_i:\ i\leq n \}\subseteq S_1 $ and $ \{ e_i:\ i\leq n \}\subseteq S_0 
$. Consider 
\[ V_n:=\left\lbrace  v_{2n+2-4k}:\ k\leq \frac{n+1}{2}\right\rbrace \cup \left\lbrace v_{2n+1-4k}:\ k\leq \frac{n}{2} \right\rbrace.  \]
(So $ V_0=\{ v_1, v_2 \} $,  $ V_1=\{ v_1,  v_2, v_5 \} $, $ V_2=\{  v_1,  v_2, v_5, v_6 \} $, $ V_3=\{v_1,  v_2, v_5, v_6, v_9  \} $ etc.) All the 
edges in $ \delta(V_n) $ but $ h_{n+1} $ are in $ S_0 $, thus 
necessarily $ h_{n+1}\in 
S_1 $. 
But then all the edges in 
$ \delta(\{ v_m:\ m\leq 2n+3\}) $ but $ e_{n+1} $ are in $ S_1 $ therefore $ e_{n+1}\in S_0 $. It follows by induction that 
\[ S_0=X_0\cup \{e_i:\  i<\omega \} \text{ and } S_1=X_1\cup \{h_i:\  i<\omega \}.  \]

But then $ S_1 $ consists of two vertex-disjoint rays contradicting the assumption that $ S_1 $ is a spanning tree.
\end{proof}

\section{Generalisations of the Aharoni-Pouzet base exchange theorem}\label{sec: positiveBaseExch}
The difficulty of base exchange problems in matroids of infinite rank is the lack of meaningful 
arithmetic operations with the rank function. Indeed, the standard proof of the 
Greene-Magnanti theorem involves subtractions of certain values of the rank function but these subtractions are no longer 
well-defined if the 
values in question are infinite. Furthermore, the contraction of a single edge in a base does not reduce the rank if it 
is infinite, thus the direct adaptation of proofs based on induction on the rank after such a contraction is also impossible. 

In contrast to the 
approach by Aharoni and Pouzet, our
proofs do not rely on infinite matching theory but use the elementary but powerful method by Kotlar, Roda and R. Ziv introduced in 
\cite{koltar2021seq}. This method 
applies symmetric 
subset base exchange as a subroutine. We are going to show that symmetric 
subset base exchange works in any matroid as long as the set we intend to exchange is finite.

\subsection{Partition base exchange with finite sets}

In contrast to  finitary matroids, it is unprovable for general matroids  that the bases must have the same size (it is independent of the axiomatic 
set 
theory 
ZFC \cite{higgs1969equicardinality, bowler2016self}) but the following weakening is easy to prove:
\begin{lem}[{\cite[Lemma 3.7]{bruhn2013axioms}}]\label{lem: eqbase weak}
If $ B_0 $ and $ B_1 $ are bases of a matroid with $ \left|B_0\setminus B_1\right| <\aleph_0$, then $ \left|B_0\setminus B_1\right| 
=\left|B_1\setminus B_0\right|$.
\end{lem}
\begin{prop}\label{prop: InfFinBipartition}
Suppose that $ M=(E, \mathcal{I}) $ is a matroid, $ B_0 $ and $ B_1 $ are bases of $ M $ and $ X $ is a finite or cofinite subset of $ B_0 $. Then 
there is a $ 
Y\subseteq B_1 $ such that $ (B_0\setminus X) \cup Y $ and $ 
(B_1\setminus Y) \cup X $ are both bases.
\end{prop}
\begin{proof}
By the symmetry between $ X $ and $ B_0\setminus X $ we may assume that $ X $ is finite. We can also assume without loss 
of generality that $ B_0\cap 
B_1=\emptyset $ since otherwise we consider the matroid $ M/(B_0\cap B_1) $, bases $ B_0\setminus B_1 $ and $ B_1\setminus B_0 $ and set $ 
X':= X\setminus B_1 $. If $ Y' $ is as desired for the new problem, then $ Y:=Y'\cup (B_0\cap B_1\cap X) $ is suitable for the original one. 

We take disjoint sets $ Y, Z\subseteq B_1 $ (see Figure \ref{fig: symmSubEx}) such that $ (B_0\setminus X)\cup Y $ and $ X\cup Z $ are both 
independent and for the 
set $ U:=B_1\setminus (Y\cup Z) $ of uncovered edges, $\left|U\right| $ is as small as possible. The set $ Z $ itself can cover all but at most $ 
\left|X\right| $ edges from $ B_1 $  and it misses exactly $ \left|X\right| $ many edges iff $ X\cup Z $ is a base 
(see Lemma \ref{lem: eqbase weak}). In particular,  $ U $ must be finite. Similarly, $ Y $ can cover 
at most $ \left|X\right| $ edges and $ (B_0\setminus X)\cup Y $ is a base iff $ \left|X\right|=\left|Y\right| $. Therefore it is enough to prove that $ 
U=\emptyset $.  Suppose for a contradiction that it is not the case. By relocating edges from $ Z $ to $ Y $, we can assume that $ 
\left|X\right|=\left|Y\right| $. 

\begin{figure}[h]
\centering

\begin{tikzpicture}[scale=0.7]

\draw  (-2.5,2.5) rectangle (2,-5);
\draw  (2,2.5) node (v1) {} rectangle (6.5,-5);
\draw  (-2.5,-5) rectangle (6.5,-3.5);
\draw  (v1) node (v2) {} rectangle (4.5,-5);
\draw  (v2) rectangle (4.5,1.5);

\node at (0,3.2) {$B_0$};
\node at (4.5,3.2) {$B_1$};
\node at (3.2,2.8) {$S$};
\node at (4.4,-5.4) {$Y$};
\node at (-0.2,-4.3) {$X$};
\node at (3.2,-4.2) {$Y\cap S$};
\node at (5.4,-4.3) {$Y\setminus S$};
\node at (3.2,2) {$U$};
\node at (-0.2,-0.8) {$B_0\setminus X$};
\node at (3.2,-0.8) {$Z\cap S$};
\node at (5.5,-0.9) {$Z\setminus S$};
\end{tikzpicture}

\caption{} \label{fig: symmSubEx}
\end{figure}
\begin{claim}
There is an $ S $ with $ U\subseteq S \subseteq B_1 $ such that $ S $ is spanned by both $ (Y\cap S)\cup (B_0 \setminus X) 
$ and $ (Z\cap S)\cup X $.
\end{claim}
\begin{proof}
We need the following classical   ``augmenting path lemma''  developed by Edmonds and Fulkerson in \cite{edmonds1968transversals}. A 
discussion of their method in the context of infinite matroids can be found for example in \cite[Subsection 3.1]{erde2019base}.
\begin{lem}[Edmonds and Fulkerson]\label{lem: augpath}
Let $ \{ M_i:\ i<\kappa \} $ be a family of matroids defined on the common edge set $ E $, let $ \{ I_i:\ i<\kappa \} $ be a family of pairwise 
disjoint sets such that~$I_i$ is independent in~$M_i$ and  $ U:=E\setminus \bigcup_{i<\kappa}I_i\neq \emptyset $. Then there is either another 
family $ \{ J_i:\ i<\kappa \} $ of pairwise disjoint sets where $ J_i $ is independent in $ M_i $ and an $ e\in U $ for which  
\[ \bigcup_{i<\kappa}J_i=\{ e \}\cup\bigcup_{i<\kappa}I_i \]   
or there is an $ S $ with $ U\subseteq S \subseteq E $ such that $ I_i\cap S $ spans $ S $ in $ M_i $ for every $ i<\kappa $.
\end{lem}

We apply Lemma \ref{lem: augpath} with the matroids $ M/X \upharpoonright B_1 $ and $ M/(B_0\setminus X) \upharpoonright B_1 $ and sets $ 
Z $ and $ Y $. By the choice of $ Y $ and $ Z $ it is impossible to cover more edges thus the second case of Lemma \ref{lem: augpath} occurs 
which provides the desired $ S $.
\end{proof}
By the properties of $ S $, the set $ X\cup  (Z\cap S) $ is a base of $ M \upharpoonright (X\cup S) $. Since $ S $ is independent in $ M $, there is 
some $ 
X'\subseteq X $, that $ X'\cup S $ is also a base of $ M \upharpoonright (X\cup S) $. Lemma \ref{lem: eqbase weak} applied to $ M 
\upharpoonright (X\cup S) $ with these two bases and $ 
\left|X\right|=\left|Y\right| $ ensure that 
 \[ \left|X'\right|=\left|X\right|-\left|Y\cap S\right|-\left|U\right|=\left|Y\setminus S \right|-\left|U\right|<\left|Y\setminus 
 S\right|. \]
 
Since $ (Y\cap S)\cup (B_0 \setminus X) $ spans $ S $ and $ X'\cup S $ spans $ X\cup S $, the set $ (Y\cap S)\cup (B_0 \setminus X)\cup X' $ 
spans $ X $ and therefore spans $ B_0 $ as well, thus contains a base. But it is ``too small'' to contain a base according to Lemma 
\ref{lem: eqbase weak} because  $ 
B_0\setminus X $ needs 
at least $ \left|X\right| $ new edges to become a base and
\[ \left|Y\cap S\right|+\left|X'\right|<\left|Y\cap S\right| +\left|Y\setminus S\right|=\left|Y\right|=\left|X\right|.\]
\end{proof}

We prove a slightly stronger statement than Theorem \ref{thm: GreeneMagnantiIntro} because we need later the extra generality to prove Theorem 
\ref{thm: ExchangeAllFin}.
\begin{thm}\label{thm: GreeneMagnanti}
Suppose that $ M=(E, \mathcal{I}) $ is a finitary matroid, $ B_0 $ and $ B_1 $ are bases  and 
$B_0=\bigcup_{i<\kappa}X_i  $ is a 
partition where all the $ X_i $ are finite. Then there is a 
partition $B_1=\bigcup_{i<\kappa}Y_i  $ and a bijection $ \sigma: \kappa \rightarrow \kappa $ such that $ (B_0\setminus X_i)\cup Y_i $ and
$ (B_0 \setminus \bigcup_{i\leq j<\kappa} X_{\sigma(j)})\cup \bigcup_{i\leq j<\kappa} Y_{j} $ are bases for every $ i<\kappa $. For $ 
\kappa \leq \omega $, the $ \sigma $ can be chosen to be the identity. 
\end{thm}
\begin{proof}
For $ \kappa<\omega $, this is \cite[Theorem 1.1.]{koltar2021seq}. Suppose  first that $ \kappa = \omega $.
Assume that the $ Y_i $ are already defined for $ i<n $ for some $ n<\omega $ such that for every $ i<n $, the sets $ (B_0\setminus X_i)\cup Y_i 
$ 
and
$ (B_0 \setminus \bigcup_{i\leq j<\omega} X_{j})\cup (B_1\setminus \bigcup_{j<i} Y_{j}) $ are bases.  We apply Proposition 
\ref{prop: InfFinBipartition} with the matroid $ M/\bigcup_{i<n} X_{i} $, bases 
$ B_0\setminus\bigcup_{i<n} X_{i}  $ and $ B_1\setminus\bigcup_{i<n} Y_{i} $, and set $ X_n $ to obtain $ Y_n $. The 
recursion is done.  Since  $ Y_n\subseteq B_1 \setminus \bigcup_{i<n} Y_{i} $ for every $ n<\omega $, the sets $ Y_i $ are pairwise disjoint. 
It follows by induction via Proposition \ref{prop: InfFinBipartition}  that $ (B_0\setminus X_i)\cup Y_i$ and
$ (B_0 \setminus \bigcup_{i\leq j<\omega} X_{j})\cup (B_1\setminus \bigcup_{j<i} Y_{j}) $ are bases for every 
$ i<\omega $. In order to show that 
$ B_1=\bigcup_{i<\omega}Y_i $, let $ e\in B_1 $ be arbitrary. Since $ C(e, B_0) $ is finite, there is an $ i<\omega $ that 
$ C(e, B_0)-e\subseteq \bigcup_{j< i}X_j $. But then we must have $ e\in \bigcup_{j< i}Y_j  $ because   $ (B_0 \setminus \bigcup_{i\leq 
j<\omega} X_{j})\cup (B_1\setminus \bigcup_{j<i} Y_{j}) $ is a base. Therefore $ B_1\setminus \bigcup_{j<i} Y_{j}=
\bigcup_{i\leq 
j<\omega} Y_{j} $ which completes the proof of the case $ \kappa=\omega $.

Suppose that $ \kappa>\omega $. We reduce this to the countable case by the following technical lemma. Let us state the lemma in a slightly more 
general form than it is actually needed (it does not change the proof and having a reference for this more general form could be helpful).
\begin{lem}\label{lem: elementarySubmodel}
Assume that $ M=(E,\mathcal{I}) $ is a matroid without uncountable circuits, 
 $ B_0 $ and 
$ B_1 $ are bases of $ M $ with $ \aleph_0<\left|B_0\right|=:\kappa $ and $ 
B_0=\bigcup_{i<\kappa} X_i $ is a partition where each $ X_i $ is countable. Then there is a family $ \{ B_j^{\alpha}:\ j\in \{ 0,1 \},\ 
\alpha<\kappa 
\}  $  such that
\begin{enumerate}
\item\label{item: ures} $ B_0^{0}=B_1^{0}=\emptyset $;
\item\label{item: unio is base} $ \bigcup_{\alpha<\kappa}B_j^{\alpha}=B_j $ for $ j\in \{ 0,1 \} $;
\item\label{item: countable pieces} $B_j^{\alpha}\subseteq B_j^{\alpha+1} $ with $ \left|B_j^{\alpha+1}\setminus B_j^{\alpha}\right|=\aleph_0 
$ for every $ \alpha<\kappa $ 
and $ j\in \{ 0,1 \} $;
\item\label{item: continues} For a limit ordinal $ \alpha<\kappa $ and $ j\in \{ 0,1 \} $ we have $ 
B_j^{\alpha}=\bigcup_{\beta<\alpha}B_j^{\beta} $;
\item\label{item: span each other} $ B_0^{\alpha} $ and $ B_1^{\alpha} $ span each other in $ M $ for every $ \alpha $;
\item\label{item: not split X_i} Whenever $ B_0^{\alpha}\cap X_i\neq\emptyset $ for some $ \alpha, i<\kappa $, then $ X_i\subseteq 
B_0^{\alpha} $.
\end{enumerate}
\end{lem}
\begin{proof}
The proof is  a straightforward transfinite recursion (or alternatively a basic application of a chain of elementary submodels). Suppose that $ 
B_0^{\alpha}$ and $ B_1^{\alpha} $ are already defined. 
We set $ B_j^{\alpha, 0}:=B_j^{\alpha} $ for $ j\in \{ 0,1 \} $ and let  $ B_0^{\alpha, 1}:=B_0^{\alpha}\cup X_i $ for the smallest $ i $ with 
$ B_0^{\alpha}\cap X_i=\emptyset $. If $ B_0^{\alpha, n+1} $ and $ B_1^{\alpha, n} $ are defined for some $ n<\omega $, then 
 \begin{align*}
 B_1^{\alpha, n+1}&:=B_1^{\alpha, n}\cup \bigcup \{C(e, B_1)-e:\  e\in B_0^{\alpha, n+1}\setminus B_0^{\alpha, n} \}\\
 B_0 ^{\alpha, n+2}&:=  B_0 ^{\alpha, n+1} \cup \bigcup \{ X_i:\ (\exists e\in B_1^{\alpha, n+1}\setminus B_1^{\alpha, n}) (C(e, B_0)\cap 
 X_i\neq \emptyset) \}. 
 \end{align*}
 
 It is easy to check that $ B_j^{\alpha+1}:=\bigcup_{n<\omega}B_j^{\alpha, n} $ for $ j\in \{ 0,1 \} $ are suitable. Since limit steps obviously 
 preserve all the conditions, we are done.
\end{proof}
Let $ B_j^{\alpha} $ for $ j\in \{ 0,1 \} $ and $ \alpha<\kappa $ as in Lemma \ref{lem: elementarySubmodel}. Properties 
(\ref{item: countable pieces}) and (\ref{item: not split X_i}) guarantee that for every $ \alpha<\kappa $, the set $ B_{0}^{\alpha+1}\setminus 
B_{0}^{\alpha} $ is the union of countably infinite $ X_i $. Let $ \sigma_\alpha $ be an $ \omega $-type enumeration  of the sets $ X_i $ that are 
contained in 
$ B_{0}^{\alpha+1}\setminus B_{0}^{\alpha} $. We choose $ \sigma $ to be the concatenation of the sequences $ \sigma_\alpha $. Properties  
(\ref{item: countable pieces}) and (\ref{item: span each other}) guarantee that $B_{0,\alpha}:= 
B_0^{\alpha+1}\setminus 
B_0^{\alpha} $ and $B_{1,\alpha}:= B_1^{\alpha+1}\setminus B_1^{\alpha} $ are bases of  $ M_\alpha:=M \upharpoonright 
(B_0^{\alpha+1}\cup B_1^{\alpha+1} ) /B_{0}^{\alpha} $. For every $ 
\alpha<\kappa $ we apply the already proved countable case with
matroid $ M_\alpha $, bases $ B_0^{\alpha} $ and $ B_1^{\alpha} $  and partition 
$ B_{0,\alpha}=\bigcup_{n<\omega} X_{\sigma_\alpha(n)} $. Let 
$B_{1,\alpha}=\bigcup_{n<\omega} Y_{\alpha, n} $ be the resulting partition. We shall prove that letting $ Y_{\omega 
\alpha+n}:=Y_{\alpha, n} $  
results in a desired partition of $ B_1 $. The sets $ B_{1,\alpha} $ for $ \alpha<\kappa $ form a partition of $ B_1 $ by the properties
(\ref{item: ures})-(\ref{item: continues}).  
The sets $ Y_{\alpha, n} $ for $ n<\omega $ partition $ B_{1,\alpha} $ by construction. Thus the sets $ Y_i $ for $ i<\kappa $ partition $ 
B_1 $. Let $ \omega \alpha+n<\kappa $ be arbitrary. By construction 
\[  (B_{0,\alpha}\setminus X_{\sigma(\omega \alpha+n)})\cup Y_{\omega \alpha+n}   \text{ and } 
 [(B_{0,\alpha}\setminus \bigcup_{n\leq m<\omega} X_{\sigma(\omega \alpha+m)})]\cup \bigcup_{n\leq m<\omega} Y_{\omega \alpha+m} \]  
are bases of $ M_\alpha $. But then their respective union with $ 
B_0^{\alpha} $  results in bases of $ M \upharpoonright (B_0^{\alpha+1}\cup B_1^{\alpha+1}) $. By property (\ref{item: span each other}), such 
bases can be extended to bases of $ M $ 
by adding any of $ B_0\setminus B_0^{\alpha+1}  $ and 
$ B_1\setminus B_1^{\alpha+1}=\bigcup_{\omega(\alpha+1)\leq \beta<\kappa}Y_\beta  $. Thus the desired exchange properties hold.
\end{proof}

From the proof above it is clear that Proposition \ref{prop: InfFinBipartition} has the following extension:
\begin{cor}\label{cor: finManyPieces}
Suppose that $ M=(E, \mathcal{I}) $ is a matroid, $ B_0 $ and $ B_1 $ are bases, $ n<\omega $  and 
$B_0=\bigcup_{i\leq n}X_i  $ is a 
partition where all but at most one $ X_i $ are finite. Then there is a 
partition $B_1=\bigcup_{i\leq n}Y_i  $ such that $ (B_0\setminus X_i)\cup Y_i $ and $ (\bigcup_{j<i}X_j)\cup \bigcup_{i<j\leq n}Y_j $ is a base 
of $ M $ for each $ i\leq n $.
\end{cor}
\subsection{Exchanging all finite subsets of a base} In this subsection we prove a common generalisation of \cite{brylawski1973some} and 
\cite[Theorem 2.1]{aharoni1991bases}. We repeat it here for convenience:
\ExchangeAllFin*
\begin{proof}
We will make use of the following special case of Theorem \ref{thm: GreeneMagnanti} where the partition consists of singletons:
\begin{cor}\label{cor: WellOrder}
Assume that $ M=(E,\mathcal{I}) $ is a finitary matroid  and $ B_0 $ and $ B_1 $ are bases. Then there are 
enumerations $ B_0=\{ e_\alpha:\ 
\alpha<\kappa \} $ and $ B_1=\{ f_\alpha:\ \alpha<\kappa \} $ such that $ B_0-e_\alpha + f_\alpha $ 
and $ (B_0 \setminus \{ e_\beta:\ \alpha\leq \beta<\kappa\})\cup \{ f_\beta:\ \alpha\leq \beta<\kappa \} $ are bases for every $ \alpha<\kappa $.
\end{cor}

It is enough to show that for every $ k<\omega $ there is a bijection $ F_k: [B_0]^{k}\rightarrow [B_1]^{k} $ for which $ (B_0\setminus 
I)\cup F_k(I) $ is a base for every $ I\in [B_0]^{k} $ because then $ F:=\bigcup_{k<\omega}F_k $ is suitable. We define $ F_0:=\emptyset $. 
Suppose that 
we already know for some $ k $ and every $ M, B_0 $ and $ B_1 $ that such a bijection $ F_k =F_{ k,M, B_0, B_1 }$ exists. Let $ M, B_0 $ 
and $ B_1 $ be fixed. We also fix enumerations as in Corollary \ref{cor: WellOrder} and let us well-order $ B_0 $ and $ B_1 $  according to these 
enumerations. In order to 
define a desired $ F_{k+1} $, it is enough to give for every $ \alpha<\kappa $ a bijection $ F_{k+1, \alpha} 
$ between the $ k $-subsets of $ B_0 $ with smallest edge $ e_\alpha $ and the $ k $-subsets of $ B_1$ with  smallest edge $ f_\alpha $. Indeed, if 
it is 
done, then $ F_{k+1}:=\bigcup_{\alpha<\kappa}F_{k+1, \alpha} $ is appropriate. Corollary \ref{cor: WellOrder} guarantees that
$ B_{0,\alpha}:=\{ e_\beta:\ \alpha < \beta<\kappa \} $ and $B_{1,\alpha}:= \{ f_\beta:\ \alpha < \beta<\kappa \} $ are bases in 
$M_\alpha:= M/(\{ e_\beta:\ \beta<\alpha \}\cup \{ f_\alpha \}) $. Let $ F_{k+1, \alpha}' $ be 
what we get by applying the induction hypothesis for $ k $ with $ M_\alpha,  B_{0,\alpha} $ and $ B_{1,\alpha} $. Then Corollary 
\ref{cor: WellOrder} and the induction hypothesis ensure that defining  \[ F_{k+1, \alpha}(I):=F_{k+1, 
\alpha}' (I-e_\alpha)+f_\alpha \] is 
suitable.
\end{proof}

\begin{quest}
Is it true  for every finitary matroid $M $ and bases  $ B_0 $ and $ B_1 $  that there exists a bijection $ F: 
\mathcal{P}(B_0)\rightarrow  \mathcal{P}(B_1)$ such that  $ (B_0\setminus I)\cup F(I) $ is a 
base for every $I\subseteq B_0 $?
\end{quest}
\printbibliography
\end{document}